\documentclass{amsart}

\usepackage{amsmath,amsthm,amsfonts,amssymb,amscd}
\usepackage[all]{xy}
\usepackage[english]{babel}
\usepackage{graphicx,color}
\usepackage[T1]{fontenc}
\usepackage[latin1]{inputenc}
\usepackage{ae}
\usepackage{enumerate}   
\usepackage[mathscr]{eucal}
\usepackage{amsmath,amssymb}
\usepackage{graphics}
\usepackage{multirow}
\usepackage{hyperref}
\usepackage[active]{srcltx}

\newtheorem{thm}{Theorem}[section]
\newtheorem{THM}{Theorem}{\Alph{THM}}

\newtheorem{prop}[thm]{Proposition}

\newtheorem{lemma}[thm]{Lemma}

\theoremstyle{definition}

\newtheorem{remark}[thm]{Remark}
\newtheorem{example}[thm]{Example}

\newcommand{\nc}{\newcommand}
\nc {\hh}{\check{h}}
\nc {\DD}{\mathcal{D}}
\nc {\RR}{\mathcal{R}}
\nc {\Pp}{\mathbb{P}}
\nc {\Ss}{\mathcal{S}}
\nc {\PP}{\mathbb{P}^{2}}
\nc {\Pd}{ \check{\mathbb{P}}^{2}}
\nc {\WW}{\mathcal{W}}
\nc {\Sym}{\mathrm{Sym}}
\nc {\OO}{\mathcal{O}}
\nc {\UU}{\mathcal{U}}
\nc {\EE}{\mathcal{E}}
\nc {\MM}{\mathcal{M}}
\nc {\KK}{\mathcal{K}}
\nc {\PW}{\mathcal{P}}
\nc {\NW}{\mathcal{N}_{\WW}}
\nc {\FF}{\mathcal{F}}
\nc {\GG}{\mathcal{G}}
\nc {\ZZ}{\mathcal{Z}}
\nc {\LL}{\mathcal{L}}
\nc {\HH}{\mathcal{H}}
\nc {\NN}{\mathcal{N}}
\nc {\VV}{\mathcal{V}}
\nc {\Ww}{\mathbb{W}}
\nc {\QQ}{\mathbb{Q}}
\nc {\II}{\mathcal{I}}

\begin{document}

\date{}
\author{M. Falla Luza}
\author{P. Sad}
\email{maycolfl@gmail.com, sad@impa.br}

\title{Positive Neighborhoods of Curves}

\maketitle

\begin{abstract}
In this work we study neighborhoods of curves in surfaces with positive self-intersection that can be embeeded as a germ of neighborhood of a curve on the projective plane.
\end{abstract}

\tableofcontents

\section{Introduction}

We study in this paper neighborhoods of compact, smooth, holomorphic curves of complex surfaces which have positive self intersection number. Our main purpuse is to give a condition that guarantees the existence of an embedding of a neighborhood of the curve into the projective plane. The first example of a result on this problem comes from \cite{FA}; in that paper the authors showed that if the curve has genus 0 and self intersection number equal to 1 then the existence of three different fibrations over it implies that some neighborhood is diffeomorphic to a neighborhood of the line in the projective plane. In this paper we consider curves of self-intersection $d^2$ with $d \geq 2$.

Since a fibration over a curve of genus 0 is defined by a local submersion over $\Pp^1$ (that is, defined in a neighborhood of the curve), we may wonder if in the case  of higher genus the existence of a  number of local submersions is enough to guarantee an embedding into the projective plane $\Pp^2$. It is in fact a necessary condition. 

In order to discuss this, let us suppose  that a curve $C$ contained in some surface $S$ can be embedded in $\mathbb P^2$ as a curve $C_0$ of degree $d \geq 2$ (we have of course to start with $C\cdot C= d^2$ in $S$). It is easy to find infinitely many submersions in a neighborhood of $C_0$. For example, we take two curves $\{A=0\}$ and $\{B=0\}$ of the same degree $l\in {\mathbb N}$ which cross each other in $l^2$ distinct points not in $C_0$. It can be seen that the map $A/B$, which is well defined outside $\{A=0\}\cap \{B=0\}$, has no multiple fibers so that it has only a finite number of critical points; if $C_0$ avoids all these points then $A/B$ is a submersion in some neighborhood of $C_0$ and the restriction of $A/B$ to $C_0$ is a ramified map from $C_0$ to $\Pp^1$ of degree $l.d$\,. We will be particularly interested in the case $l=1$, that is, $A=0$ and $B=0$ are lines whose common point  is not in $C_0$; the submersion $A/B$ will be called a {\it pencil submersion} and the restriction of $A/B$ to $C_0$ is a ramified map of degree $d$ (any local submersion that leaves such a trace in $C_0$ is in fact a pencil submersion). We see that to be  equivalent to a neighborhood of $C_0$, a neighborhood of $C$ has to carry also  many submersions to $\Pp^1$, the surprising feature in \cite{FA} is that only three submersions are needed. The converse is not true as we can see in the following example. 

\begin{example}\label{fake-example}
Consider the rational curve in $\mathbb P^2$ defined in affine coordinates by the equation $y^2=x^2(x+1)$; it is a smooth rational  curve except for the node at the point $(0,0)$. We blow up first at a point in the curve different from $(0,0)$, and then we blow up at $(0,0)$. The strict transform is a smooth rational curve $C$ of self intersection number equal to 4 with many local submersions (which come from submersions constructed in the plane as above), but its neighborhood can not be embedded in the plane: given a submersion constructed using $l=1$ as above (before blow up's), we notice that it induces a ramified map from $C$ to $\Pp^1$ of degree 3; but for a  conic $C_0$ in the plane (which has of course self intersection number equal to 4), the ramified map induced by any local submersion is of even degree. 
\end{example}

A more refined question would be: can we obtain an embedding once it is assumed the  existence of three local submersions in a neighborhood of $C$ whose restrictions to $C$ are meromorphic maps of degree (a multiple of) $d$? We give a partial negative answer in Section \ref{sec-examples}.
 
We introduce then an extra condition (also a necessary one). A curve $C$ that has an embedding $\phi:C \rightarrow C_0\subset \mathbb P^2$ carries naturally a special set   of meromorphic maps ${\bf G}_{\phi}$ =$\{G_{|_{C_0}}\circ {\phi},\,\, G \,\,\,pencil\, submersion\}$. A set $\{F_i\}$ of submersions defined in a neighborhood of $C$ whose restrictions to $C$ have no common critical points is {\bf projective at C} if ${F_i}_{|_C}\in {\bf G}_{\phi}$. The submersions are called {\bf independet} if the singularities of the correspondent pencils on $\Pp^2$ are not aligned. We may state then our main result:

\begin{THM}\label{main-thm}
The existence of a projective triple of independent submersions at $C$  implies the existence of an embedding of a neighborhood of $C$ into the projective plane.
\end{THM}

The submersions in the statement of the Theorem are supposed to produce different fibrations; we remark that if $F$ is a submersion over  $\Pp^1$ and $T$ is a Moebius transformation, then $F$ and $T\circ F$ induce the same fibration.

\begin{remark}
The fibers of a submersion define a regular foliation in a neighborhood of $C$, which is generically transverse to $C$ with tangency points at the critical points of the restriction to the curve; the submersion is a meromorphic first integral for the foliation. The converse does not hold, that is, this type of foliation may not have a first integral, see \cite{MEZ}.
\end{remark} 

We mention that the study of neighborhoods of curves has already been pursued when the self-intersection is not positive as we can see in \cite {GRA}, \cite {SAV} and \cite{UE}.

This paper is organized as follows:  Section \ref{sec-examples}  presents  some examples and it is followed by  Section \ref{sec-const-mer-maps} where we discuss how to built meromophic maps starting from two different pencil submersions. This allows (Section \ref{sec-new-foliation}) to show the existence of foliations defined in a neighborhood of the curve   which have this curve as an invariant set, and finally in Section \ref{sec-proof-thm} we prove our theorem.

\section{Examples}\label{sec-examples}
This Section has two parts. In the first part we give examples  of surfaces containing  smooth curves   of self-intersection number $d^2$ which are not embeddable in the plane, although they are fibered by  submersions whose restrictions to the curves are meromorphic functions of a degree multiple of $d$. Once this is done, we give examples which satisfy the extra condition of our Theorem but have only one or two fibrations and  do not  embed them in the plane.

\subsection{Separating branches and examples with 3  fibrations}
 
We will use the following construction. Let us consider a curve H with an ordinary singularity $P$ with $m$ branches $L_1,\dots,L_m$. For each branch $L_j$ we take a neighborhood $V_j$ which is biholomorphic to a bidisc $D_j$ by means of a biholomorphism $\phi_j:D_j\rightarrow V_j$ ;  we assume that $\delta L_j \cap V_i =\emptyset$ for all $i\ne j$. We fix a neighborhood $V$ of $H\setminus \cup_1^m L_j$. Finally we take the disjoint union of $V$ with all th $D_j$, and glue $D_j$   to $V$ using the restriction of the map $\phi_j$ to $\phi_j^{-1}(V\cap V_j)$. In this way the union  of the sets $V_j$, which contains  $P$, is replaced by $m$ copies of the bidisc, and there is a new curve $H^{\prime}$ replacing $H$ inside a new surface without the ordinary singularity. As for the self-intersection number $H^{\prime}\cdot H^{\prime}$, we have that $H^{\prime}\cdot H^{\prime}= H\cdot H -m(m-1)= (H\cdot H -m^2) +m$.
Also any holomorphic foliation $\mathcal F$ defined in $V\cup_1^m V_j$ induces naturally a holomorphic foliation in the new surface which is $\mathcal F$ in $V$ and $\phi_j^*(\mathcal F)$ in each $D_j$. We refer to this construction as {\it separating branches of $H$  at P}.
\begin{figure}[ht!]
  \centering
    \includegraphics[scale=0.6]{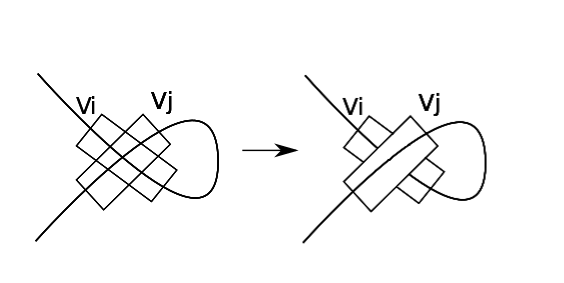}
  \caption{Separating branches}
  \label{fig:viaduto}
\end{figure}

Let us consider then a smooth plane curve $C^{\prime}$ of degree $d^{\prime}$ and genus $g(C^{\prime})= \dfrac{(d^{\prime}-1)(d^{\prime}-2)}{2}$; it can be also immersed in the plane as a curve $C$ of degree $d$ for any $d>2g(C^{\prime})$ with a number $s$ of nodal points such that $d^2-3d-2s={d^{\prime}}^2 -3d^{\prime}$. We choose $d-d^{\prime}=k^2$ for some $k\in \mathbb N$ such that 
\begin{enumerate}[(i)]
\item $d^{\prime}$ divides $k^3-k^2$,
\item $d^{\prime}$ does not  divide $2k^2$;
\end{enumerate}
we choose also  three pencils $d\left(\dfrac {u_j}{v_j}\right)=0$ of curves of degree $k$ whose sets of $k^2$ base points  lie in the regular part of $C$ and are two by two disjoint. After blowing up at these $3k^2$ points and separating branches at the nodal points of $C$ we get a curve $\tilde C$ containing in some surface with self-intersection number equal to $d^2-3k^2-2s = {d^{\prime}}^2$; the maps $\dfrac {u_j}{v_j}$ become submersions whose restrictions to $\tilde C$ are meromorphic maps of  degree $k.d- k^2$, which is a multiple of $d^{\prime}$ because of (i).

 A neighborhood of  $\tilde C$ is not equivalent to a neighborhood of $C^{\prime}$ in the plane. In fact, let us take a linear pencil $\mathcal L$ in the plane with base point outside $C$ and transverse to the branches at each nodal point (this is before blow-up`s and separation of branches). We have $2d=Tang(\mathcal L,C)+ \chi(C)$; since $tang(\mathcal L,C,P)=2$ for each nodal point $P$, we get $2d = 2s+\chi(C)+ tang(\mathcal L,C)$, where the last term counts the tangencies with the regular part of $C$. These tangencies persist when we blow up and separate branches; therefore, if a neighborhood of $\tilde C$ is equivalent to a neighborhood of $C^{\prime}$, we get in this neighborhood a foliation $\mathcal L^{\prime}$ with $Tang(\mathcal L^{\prime}, C^{\prime})=tang(\mathcal L,C)=2d-2s-\chi(C)$. It follows that $(deg(\mathcal L^{\prime})+2)d^{\prime}=2d-2s-\chi(C)+\chi(C^{\prime})=2d-2s+2g(C)-2g(C^{\prime})$ and since $g(C)-s=g(C^{\prime})$, we conclude that $(deg(\mathcal L^{\prime})+2)d^{\prime}=2d= 2d^{\prime}+2k^2$, a contradiction because of (ii).
 
We remark that when $d=4, d^{\prime}=3$ or $d=3,d^{\prime}=2$ the construction can be done with $k=1$ because $d>2g(C^{\prime})$ (and obviously (i) is satisfied in both cases). When $d=4, d^{\prime}=3$ we have also that (ii) holds true. In the general case $d^{\prime}>3$ we may choose $k=d^{\prime}+1$ for example in order to get both (i) and (ii) satisfied.
  
The special case $d=3,d^{\prime}=2$ (and $k=1$) can be treated with a small difference in what concerns the proof that the neighborhood of $\tilde C$ is not equivalent to a neighborhood of $C$: we   select $\mathcal L$ as the pencil whose base point is the node point $P$ of $C$. Since $Tang(\mathcal L,C)= 6=tang(\mathcal L,C,P)$, we see that there is no other point of tangency between $\mathcal L$ and $C$. We get then $(deg(\mathcal L^{\prime})+2).2= 4+2$ (after separating the branches at $P$ we obtain 2 radial singularities belonging to $\tilde C$ and a fortiori to $C^{\prime}$) and therefore $deg(\mathcal L^{\prime})=1$. But is impossible for a foliation of degree 1 in the plane to have 2 radial singularities.

It would be nice to have   examples where the degree induced by the submersions on the curve is  exactly $d^{\prime}$.

\subsection{Special Examples}

A construction  already presented in \cite{SA} of a pair (curve, surface) with a submersion whose restriction to the curve is a ramification map given a priori (we will say that the submersion is a {\bf lifting} of a ramification map). We start with a line bundle of Chern class $n\in \mathbb N$ over a curve $C$; the lines of the bundle define a foliation $\mathcal L$ in the total space of the bundle. Let $f: C \rightarrow \Pp^1$ be a ramified map with simple critical points and let $p$ be one of these points. There exists an involution $i$ defined in a neighborhood of $p$ in $C$ by $f(q)= f(i(q))$ for $q$ close to $p$.

We fix a neighborhood $U$ of $p$ and a holomorphic diffeomorphism $H:U \rightarrow {\mathbb D}\times {\mathbb D}$ such that: 1) $H(p)= (0,0)$; 2) $H(C \cap U)=\{(z_1,0)\in {\mathbb D}\times {\mathbb D}\}$; 3)$H$ takes $\mathcal L$ to the foliation $dz_1=0$ and 4) $h:= H|_{C \cap U}$ conjugates $i$ to the involution $z \longmapsto -z$, that is: $h(i(q))=-h(q)$. We take also a biholomorphism $\psi$ from ${\mathbb D}\times {\mathbb D}$ to a neighborhood of $(0,0)$ with the properties: 1) $\psi (0,0)=(0,0)$; 2) $\psi(z_1,0)= (z_1,0)$; 3) $\psi(\{1/2<|z_1|<1\}\times {\mathbb D})$ is saturated by leaves of the foliation $dZ_2-Z_1dZ_1=0$ and 4) $\psi$ is a holomorphic diffeomorphism when restricted to $\{1/2 < |z_1| <1\}\times {\mathbb D}$ that sends the foliation  $dz_1=0$ to the foliation $dZ_2-Z_1dZ_1=0$. Put $H_1= \psi \circ H$.  
 
We remove from the total space of the line bundle the fibers over the points of $h^{-1}(\{|z_1|\le 1/2\})$ and glue $\psi({\mathbb D}\times {\mathbb D})$ to the remaining set using $H_1$. In this way we get a new holomorphic surface which contains $C$ (the same curve we started with) and a holomorphic foliation transverse to $C$ except at $p$, where the tangency is simple. Furthermore, the "local holonomy" of the new foliation at $p$ is exactly $i$. We repeat the same procedure for all critical points of $f$. At the end, we have a holomorphic surface that contains $C$ and a holomorphic foliation which is transverse to $C$ except at the critical points of $f$; we may even assume that the self-intersection number of $C$ is $n\in \mathbb N$. The map $f$ can be extended along the leaves (because of its compatibility with the involutions involved), producing  the desired lifting. A similar construction can be made if the critical points are not simple.

Let us give two examples of pairs (curve, surface) which are not embeddable in the projective plane.

\begin{example}
We have already noticed that, in order to be embeddable in the projective plane, all the submersions defined in the neighborhood of the curve must have as restrictions maps whose  degrees are multiple of $d$ (here  $d^2$ is the self-intersection number of the curve in the surface). This does not happen in the example given in the Introdution. We give now another example of a different nature. Take $C\subset \mathbb P^2$. We start by claiming that there exists a ramification $f:C\rightarrow \Pp^1$ of degree $(d-1)d$ such that the set of poles is not contained in any curve of degree $d-1$. In order to see this, let us start with a ramification map $f_0:C\rightarrow \Pp^1$ defined as the restriction of $\dfrac {1}{Q_0}$ to $C$, where $Q_0$ is a polynomial of degree $d-1$ which intersects $C$ transversely at $l=(d-1)d$  different points $P_1,\dots,P_l$. Let us consider nearby points $P_1^{\prime},\dots,P_l^{\prime}$ and apply Riemann-Roch's theorem  to $D=P_1^{\prime} +\dots +P_l^{\prime}$: $l(D)\ge (d-1)d-g+1$; if we want to have $l(D)>1$, we ask for $(d-1)d-g+1>1$, or $(d-1)d >\dfrac{(d-1)(d-2)}{2}$, which is always true when $d>1$. In fact, from the proof of Riemann-Roch's theorem , since $(P_1^{\prime},\dots,P_l^{\prime})$ is close to $(P_1,\dots,P_l)$, we may choose a meromorphic function close to $f_0$, so its polar divisor is $D$. On the other hand, the points $(P_1^{\prime},\dots,P_l^{\prime})$ which belong to a curve of degree $d-1$ are contained in a subvariety of dimension $\dfrac{d(d+1)}{2}-1$, and all we have to do is check if $(d-1)d > \dfrac{d(d+1)}{2}-1$, which is obvious if $d \ge 3$ (we remark that there are not two different curves of degree $d-1$ passing through the $(d-1)d$ points $P_1^{\prime},\dots,P_l^{\prime}$). We select then
$P_1^{\prime},\dots,P_l^{\prime}$ outside this subvariety in order to get the ramification map $f$ and take a lifting $F$ defined in a surface $S$. We prove then the statement: there is no embedding $\Phi: S\rightarrow {\mathbb P^2}$. In fact, the submersion $F\circ{\Phi^{-1}}$ defined in a neighborhood of $C_0 \subset {\mathbb P^2}$ extends to  $\mathbb P^2$ as a meromorphic function (holomorphic in a neighborhood of $C_0$). We observe that, for $d\ge 3$, given two embeddings $\phi_i:C\rightarrow {\mathbb P^2}$, i=1,2 there exists an automorphism $T\in Aut(\mathbb P^2)$ such that $T(\phi_1(C))= \phi_2(C)$, see Appendix. Then the map $\Phi|_{C}:C \rightarrow C_0$ comes from a linear map on $\Pp^2$ and poles of $f$ are the intersection of $C$ with a curve of degree $d-1$, which is impossible.
\end{example}

\begin{example} 
We present now an example of a non embeddable pair (curve, surface) with a set of two fibrations which is projective at the curve. We start with a projective, smooth curve $C$ and select two pencil submersions; let $f_1$ and $f_2$ be the associated ramification maps of $C$. The tangencies between the pencils are obviously pieces of the common line. We will replace one of these pieces by a non-invariant curve of tangencies between two new foliations. The idea is the same used above to realize ramification maps;  the homeomorphism $\psi$ is going to be changed. The point $p$ this time is a point of tangency, and the coordinate chart $H$   sends the foliations associated to the submersions to two foliations $(dz_1=0, \mathcal H)$. We consider in $\mathbb C^2$ a couple of foliations  ($dZ_1=0, \mathcal H^{\prime}: d(Z_1-Z_2(Z_2-Z_1))=0)$, which have $Z_1=2Z_2$ as non-invariant line of tangencies. The homeomorphism $\psi$ is choosed in order to satisfy: 1) $\psi(0,0)=(0,0)$ and $\psi(z_1,0)=(z_1,0)$; 2) $\psi|_{\{1/2<|z_1|<1\}\times \mathbb D}$ is a holomorphic diffeomorphism over its image that sends $(dz_1=0,\mathcal H)$ to $(dZ_1=0, \mathcal H^{\prime})$. We put again $H_1=\psi \circ H$,which is the new glueing map. We can see that $C^2$ does not change and so the germ of surface is not isomorphic to $(C, \Pp^2)$. 

We could also use in the construction the pair of foliations ($dZ_1=0, \mathcal H^{\prime}: d(Z_1-Z_2(Z_2-Z_1^{k+1}))=0)$ for $k\in \mathbb N$, but the curve $C$ will have self-intersection number equal to $d^2-k$.
\end{example}

\section{Constructing meromorphic maps}\label{sec-const-mer-maps}

Let us once more describe the setting we are going to analyse. We have a curve $C$ contained in some surface $S$ with $C\cdot C= d^2$ and  $d\in \mathbb N$. There exist three submersions $F$, $G$ and $H$ defined in $S$ and taking values in $\Pp^1$ which define foliations $\mathcal F$, $\mathcal G$ and $\mathcal H$ generically transversal to $C$ whose leaves are the levels curves. In order to simplify the exposition, we assume that all tangencies with $C$ are simple and distinct (when we look to the tangencies for any pair of foliations). We denote $f=F|_{C}$, $g=G|_{C}$ and $h=H|_{C}$, all of them ramification maps from $C$ to $\Pp^1$ whose ramification points correspond to the tangency points of the foliations (because $F$, $G$ and $H$ are submersions). Furthermore, we assume that $C$ embedds into $\mathbb P^2$ by a map $\phi:C\rightarrow C_0$; $C_0$ is a smooth algebraic curve of degree $d$. In order to complete the picture, we select pencil submersions $F_0$, $G_0$ and $H_0$ (with associated foliations ${\mathcal F}_0$, ${\mathcal G}_0$ and ${\mathcal H}_0$), with singular points not aligned, which restric to $C_0$ as $d$ to $1$ maps $f_0$, $g_0$ and $h_0$ to $\Pp^1$ and ask $\{f,g,h\}$ to be conjugated by $\phi$ to $\{f_0,g_0,h_0\}$: $f_0 \circ \phi = f$, $g_0 \circ \phi = g$ and $ h_0 \circ \phi = h$. We remark that $\phi(tang(\mathcal F,C))= tang (\mathcal F_0,C_0)$ once more because these tangency points are exactly the ramification points of $f$ and $f_0$ (we have also that $\phi(tang(\mathcal G,C))= tang (\mathcal G_0,C_0)$ and $\phi(tang(\mathcal H,C))= tang (\mathcal H_0,C_0)$). For simplicity, we will assume that $F_{|_C}, G_{|_C}$ and $H_{|_C}$ have only simple critical points.

\begin{lemma} 
Any pair of foliations defined by projective submersions at $C$ are generically transverse to each other along $C$.
\end{lemma}
\begin{proof}
Let $\mathcal F$ and  $\mathcal G$ be two projective submersions at $C$. From \cite{BRU} we have
$$
tang(\mathcal F,\mathcal G)\cdot C= N_{\mathcal F} \cdot C + N_{\mathcal G}\cdot C + K_S \cdot C
$$
where $tang(\mathcal F,\mathcal G)$ is the curve of tangencies between the foliations, $N_{\mathcal F}$ (resp. $N_{\mathcal G}$) is the normal bundle associated to $\mathcal F$ (resp. $\mathcal G$) and $K_S$ is the canonical bundle of $S$. Since 
\begin{eqnarray*}
N_{\mathcal F}\cdot C &=& \chi(C)+ tang ({\mathcal F},C)= 3d-d^2 + d^2-d\\
N_{\mathcal G}\cdot C &=& \chi(C)+ tang ({\mathcal G},C) =3d-d^2 + d^2-d\\
-K_S \cdot C_0 &=& \chi(C) + C\cdot C = 3d-d^2 + d^2\\
\end{eqnarray*}
we conclude that 
$
tang(\mathcal F,\mathcal G)\cdot C = d
$
so that $\mathcal F$ and $\mathcal G$ are not tangent to each other along $C$.
\end{proof}
We observe that the Lemma is not true for $d=1$ (see \cite {FA}).
 
In this Section we will see how to associate to a pair of submersions, say $F, G$, a meromorphic map $\Phi_{F,G}$. It is defined initially as a biholormorphism from a neghborhood of the set $C\setminus (A \cup \phi^{-1}(A_0))$ to a neighborhood of $C_0\setminus (A_0 \cup \phi(A))$, where $A= tang(\mathcal F, C)\cup tang (\mathcal G,C) \cup (tang(\mathcal F, \mathcal G)\cap C)$ and $A_0= tang(\mathcal F_0, C_0)\cup tang (\mathcal G_0,C_0) \cup (tang(\mathcal F_0, \mathcal G_0)\cap C_0)$. Given a point $p\in C\setminus (A \cup {\phi}^{-1}(A_0))$, the foliations $\mathcal F$ and $\mathcal G$ are transverse to each other and to $C$ in a neighborhood of this point and the foliations $\mathcal F_0$ and $\mathcal G_0$ are  transverse to each other and to $C_0$ in a neighborhood of $\phi(p)$; therefore, for $q\in S$ close to $p$ we may associate the points $q_{\mathcal F}$ and $q_{\mathcal G}$ where the leaves of $\mathcal F$ and $\mathcal G$ intersect $C$. The leaves of $\mathcal F_0$ and $\mathcal G_0$ through $\phi(q_{\mathcal F})$ and $\phi(q_{\mathcal G})$ will intersect (by definition) at the point $\Phi_{F,G}(q)$. It can be seen that this maps extends biholomorphically  to the points of $tang(\mathcal F,C)$ and $tang(\mathcal G,C)$, essentially because the foliations $\mathcal F$ and $\mathcal G$   are transverse to each other at those points. From now on we change  $A$ and $A_0$ to $A= tang(\mathcal F, \mathcal G)\cap C$ and $A_0=tang(\mathcal F_0, \mathcal G_0)\cap C_0$ and analyse the behavior of $\Phi_{F,G}$  at   points of  $A \cup {\phi^{-1}}(A_0)$. We distinguish two cases
\begin{enumerate}[(A)]
\item  $\phi(p)\in tang(\mathcal F_0,\mathcal G_0,C_0)$. 
\item $\phi(p)\notin tang(\mathcal F_0,\mathcal G_0,C_0)$.
\end{enumerate}

\begin{prop}\label{extension}
$\Phi_{F,G}$ extends meromorphically to a neighborhood of $C$.
\end{prop}
\begin{proof} 
{\bf Case A:} We may assume, choosing conveniently the coordinates $(x,y)$ around $p$ and affine coordinates $(X,Y)$,  that

\begin {itemize}
\item $p=(0,0)$, $C$ is $y=0$ and $\mathcal F$ is defined by $dx=0$;  
\item $\phi(p)=(0,0)$, $\mathcal F_0$ is defined by $dX=0$, $\mathcal G_0$ is the radial pencil with $(0,1)$ as base point ($X=0$ is a common fiber of $\mathcal F_0$ and $\mathcal G_0$);
\item $C_0$ is defined by $Y=h(X)$ with $h(0)=0$, $h^{\prime}(0)=0$ and $\phi(x)= (x,h(x))$.
\end {itemize}

The leaf of $\mathcal F$ (respec. $\mathcal G$) through a point $(x,y)$ crosses the $x$-axis at $x$ (respectively $\xi(x,y)$ for a holomorphic function $\xi$ such that $\xi(x,0)=x$.   It follows that 
$$ 
\Phi_{F,G}(x,y)= \left(x,1-\dfrac{x(1-h(\xi(x,y))}{\xi (x,y))}\right)= \left(x,\dfrac{u(x,y)}{\xi(x,y)}\right)
$$

\noindent The expression defines a meromorphic map in a neighborhood of $(0,0)$. There are two possible cases: 

\begin{itemize}
\item $\bf A_1$: the germs $x$ and $\xi$ are relatively prime; the line of poles of $\Phi_{F,G}(x,y)$ is   $\xi (x,y)=0$ and has multiplicity 1. We write $\xi(x,y)-x=y\,A_1(x,y)$  for some holomorphic function $A_1(x,y)$; the $\mathcal G$-fiber may be transversal to the $\mathcal F$-fiber (when $A_1(0,0)\neq 0$) or tangent to it (in which case $A_1(0,0)\neq 0$).
\item $\bf A_2$: the germs $x$ and $\xi$ have a common factor; write $\xi(x,y)=x(1+y\,A_2(x,y))$, thus $\Phi_{F,G}(x,y)$ is a holomorphic map ($\mathcal F$ and $\mathcal G$ have $x=0$ as a common fiber), but it may be non-injective (unless $A_2(0,0)\neq 0$). 
\end {itemize} 

{\bf Case B:} We assume:
\begin {itemize}
\item $p=(0,0)$, $C$ is $y=0$ and $\mathcal F$ is defined by $dx=0$;  
\item $\phi(p)=(0,0)$, $\mathcal F_0$ is defined by $dX=0$ and  $\mathcal G_0$ is defined by $dY-dX=0$ (in affine coordinates);
\item $C_0$ is defined by $Y=h(X)$ with $h(0)=0$, $h^{\prime}(0)=0$ and $\phi(x)= (x,h(x))$.
\end{itemize}

We have then
$$
\Phi_{F,G}(x,y)= (x,\xi(x,y)-x +h(\xi(x,y))
$$
It follows that $\Phi_{F,G}$ is a holomorphic map in a neighborhood of $p$\,; writing   $\xi(x,y)-x=y\,B(x,y)$, we see that  $\Phi_{F,G}$ is a local biholomorphism when $B(0,0)\neq 0$, that is,
the fibers of $\mathcal F$ and $\mathcal G$ are transversal at $p$.
\end{proof}
An important consequence  for us is that the pull-back by $\Phi_{F,G}$ of a holomorphic foliation $\mathcal L$ on $\Pp^2$ is also a holomorphic foliation in $S$. In the next Section we describe the singularities of  ${\Phi^*_{F,G}}(\mathcal L)$.  

\section{New foliations on S}\label{sec-new-foliation}

Let us take a foliation $\mathcal L$ on $\mathbb P^2$ defined by $\omega= LdP-d.PdL=0$, where $P(X,Y)=\sum_{i+j \leq d} a_{ij}X^iY^j$ is a polynomial of degree $d$ such that $C_0=\{P=0\}$ (we may assume $a_{0d}\neq 0$) and $L$ is a linear polynomial such that $L=0$ is transverse to $C_0$ . The singularities of $\mathcal L$ contained in $C_0$ are supposed to be disjoint of $A_0\cup \phi(A)$.

We proceed to compute the multiplicity $Z(\mathcal L^{*},C,p)$  along $C$  of $p$ as a singularity of $\mathcal L^{*}=\Phi_{F,G}^{*}(\mathcal L)$ at the points where $\Phi_{F,G}$ maybe fails to be a biholomorphism. In order to make the computation easier, we take $L(X,Y)=X+b$. 
\begin{prop}\label{Indices-Z}
With notation of the proof of Proposition \ref{extension}, we have
\begin{itemize}
\item Case A1: $Z(\mathcal L^{*},C,p)= d+mult_0(A_1(x,0))$.
\item Case A2: $Z(\mathcal L^{*},C,p)= mult_0(A_2(x,0))$.
\item Case B: $Z(\mathcal L^{*},C,p)=mult_0(B(x,0))$.
\end{itemize}
\end{prop}
\begin{proof}
{\bf Case A1}: $x$ and $\xi$ are relatively prime. It follows that
$$
P(\Phi_{F,G}(x,y))=\dfrac{yv(x,y)}{\xi(x,y)^d}
$$
In fact, $P(\Phi_{F,G}(x,0))=0$ and
$
P(X,Y)=a_{0d}Y^d + \sum_{j\leq d-1}a_{ij}X^iY^j
$
and therefore
$$
P(\Phi_{F,G}(x,y))= a_{0d}\dfrac{u^d}{\xi^d}+ \dfrac{\sum_{d-j\geq 1}a_{ij}x^iu^j{\xi}^{d-j}}{\xi^d}
$$
In particular, $v(x,0)=x^{d-1}A_1(x,0)+\dots$. We have also $L(\Phi_{F,G}(x,y))=x+b$, $b\neq 0$, so that
$$
\Phi_{F,G}^{*}\,\omega = \dfrac{1}{\xi^{d+1}}[(x+b){\xi}(yd\,v + vd\,y) - d.yv((x+b)d\xi + {\xi}d\,x)]
$$
Therefore  $\mathcal L^{*}$ is defined by $ (x+b){\xi}(yd\,v + vd\,y) - d.yv((x+b)d\xi + {\xi}d\,x)=0$ near the point $p$ and  
$$
Z(\mathcal L^{*},C,p)= 1+mult_0(v(x,0))= d+mult_0(A_1(x,0))
$$
We observe that $Z(\mathcal L^{*},C,p)>0$ when the case ${\bf A1}$ is present.

{\bf Case A2}: $\xi$ divides $x$ ($\mathcal F$ and $\mathcal G$ share the  leaf passing through $p$). Let us write as before $\xi(x,y)= x(1+yA_2(x,y))$; it follows that 
$$
\Phi_{F,G}(x,y)=(x, \dfrac {yA_2(x,y) +h(\xi(x,y))}{1+yA_2(x,y)})
$$
Writing $P(\Phi_{F,G}(x,y))= yv(x,y)$, we see that $v(x,0)=A_2(x,0)+\dots$ and 
$$
\Phi_{F,G}^{*}\,\omega = (x+b)(vdy + ydv) - d.yvdx
$$ 
We conclude that  
$$
Z(\mathcal L^{*},C,p)=mult_0(v(x,0))=mult_0(A_2(x,0))
$$
Let us notice that $Z(\mathcal L^{*},C,p)=0$ implies that $A_2(0,0)\neq 0$, that is, $\Phi_{F,G}(x,y)$ is a local biholomorphism at $p$.

{\bf Case B}: $\phi(p)\notin tang(\mathcal F_0,\mathcal G_0) \cap C_0$. We have

$$
\Phi_{F,G}(x,y)=(x,\xi-x+h(\xi(x,y))
$$
Writing $P(\Phi_{F,G}(x,y))= yv(x,y)$, we see that $v(x,0)=B(x,0)+\dots$ and
$$
\Phi_{F,G}^{*}\,\omega = (x+b)(vdy + ydv) - d.yvdx
$$
We conclude that  
$$
Z(\mathcal L^{*},C,p)=mult_0(v(x,0))_0=mult_0(B(x,0))
$$
Again, $Z(\mathcal L^{*},C,p)=0$ implies that $B(0,0)\neq0$, that is, $\Phi_{F,G}(x,y)$ is a local biholomorphism at the point p.
\end{proof}  
We intend now to see the implications of having two maps $\Phi_{F,G}$ and $\Phi_{F,H}$ simultaneously; the fibrations $\mathcal F$, $\mathcal G$ and $\mathcal H$ are associated to pencil submersions $\mathcal F_0$, $\mathcal G_0$ and $\mathcal H_0$. Let us call $B= tang(\mathcal F,\mathcal H)\cap C$ and $B_0= tang(\mathcal F_0,\mathcal H_0)\cap C_0$. We consider two foliations $\mathcal I$ and $\mathcal L$ on $\mathbb P^2$ like before. Remark that $Z(\mathcal{I}, C_0)= Z(\mathcal{L}, C_0)=d$. We will assume: 1) all singularities of $\mathcal I$ and $\mathcal L$ lie outside the set $K= A_0 \cup \phi(A)\cup B_0 \cup \phi(B)$; 2) all curves of tangencies between $\mathcal I$ and $\mathcal L$ cross $C_0$ outside the set $K$. We denote $\mathcal I^*= \Phi_{F,G}^*(\mathcal I)$ and $\mathcal L^*= \Phi_{F,H}^*(\mathcal L)$. We will use again the formulae  from \cite{BRU} to compute numerical invariants associated to tangent lines between two foliations. We have:
$$
tang(\mathcal I,\mathcal L)\cdot C_0= N_{\mathcal I} \cdot C_0 + N_{\mathcal L}\cdot C_0 + K_{\mathbb P^2}\cdot C_0=2d^2 -d,
$$
since $\mathcal I$ and $\mathcal H$ have degree $d-1$ and $K_{\mathbb P^2}\cdot C_0 =-3d$.

Let us call $\mathcal Z_1(\mathcal I^*,C)$ ($\mathcal Z_1(\mathcal L^*,C))$  the set of  points   where  $\Phi_{F,G}$   is  not a local biholomorphism (respectively $\Phi_{F,H}$ is not a local biholomorphism). We define $Z_1(\mathcal I^*,C)$ as the sum of all indexes $Z(\mathcal I^*,C,p)$ at points of $\mathcal Z_1(\mathcal I^*,C)$ (we put $Z_1(\mathcal L^*,C)$ for the correspondent sum at points of $\mathcal Z_1(\mathcal L^*,C)$).  

As for the foliations $\mathcal I^*$ and $\mathcal L^*$, we have that 
\begin{eqnarray*}
tang(\mathcal I^*,\mathcal L^*)\cdot C&=& N_{\mathcal I^*} \cdot C + N_{\mathcal L^*}\cdot C + K_S \cdot C \\
&=& Z(\mathcal I^*,C) + Z(\mathcal L^*,C) + 2d^2-3d\\
&=& Z_1(\mathcal I^*,C) + Z_1(\mathcal L^*,C) + 2d^2-d
\end{eqnarray*}
We conclude therefore that
$$
tang(\mathcal I^*,\mathcal L^*)\cdot C = Z_1(\mathcal I^*,C) + Z_1(\mathcal L^*,C)+ tang(\mathcal I,\mathcal L)\cdot C_0
$$
Observe that curves $C$ and $C_0$ appear as components of the tangency locus in both sides of the last equation, thus we cancel $d$ from the equation and consider, from now, tangency loci besides $C$ and $C_0$. This formula suggests that  $tang(\mathcal I^*,\mathcal L^*)\cap C$ may be also computed looking at the points of  $\phi^{-1}(tang(\mathcal I, \mathcal L)\cap C) \cup  \mathcal Z_1(\mathcal I^*,C) \cup \mathcal Z_1(\mathcal L^*,C)$. 

Our aim is to prove that $\Phi_{F,G}$ and $\Phi_{F,H}$ are everywhere local biholomorphisms. First of all we have to associate the tangencies between $\mathcal I$ and $\mathcal L$  to   tangencies between $\mathcal I^*$ and $\mathcal L^*$. There is a little difficulty here because $\mathcal I^*$ and $\mathcal L^*$ are obtained from $\mathcal I$ and $\mathcal L$ using different pull-back's; the pre-image by $\phi$ of a point of tangence between $\mathcal I$ and $\mathcal L$ might not be a point of tangency between $\mathcal I^*$ and $\mathcal L^*$. We take the foliations $\mathcal I$ and $\mathcal L$ defined by the equations $LdP-d.PdL=0$ and $(L+ a)dP + d.PdL=0$; their curve of tangencies is defined by $dL\wedge dP=0$ (besides the curve $C_0$). When intersecting with $C_0$, these are the points of tangency of $\mathcal F_0$ with $C_0$.

\begin{prop}\label{intersection-at-tangencies}
Let $\phi(p)$ be a point of tangency between $\mathcal F_0$ and $C_0$. Then $(tang(\mathcal I^*,\mathcal L^*), C)_p=1$.
\end{prop}
\begin{proof}
We may take local coordinates $(x,y)$ around $p$ and $(r,s)$ around $\phi(p)$ such that
\begin {itemize}
\item $C=\{y=0\}$ and $C_0= \{s=0\}$.
\item $\mathcal F$ and $\mathcal F_0$ are defined by $d(y-x^2)=0$ and $d(s-r^2)=0$ respectively.
\end {itemize}
The foliations $\mathcal I$ and $\mathcal L$ are defined as $ su\,d(s-r^2)-(s-r^2 + \delta)\,d(su)=0$ and $su\,d(s-r^2)-(s-r^2 + a+\delta)\,d(su)=0$, where $u$ is a holomorphic function such that $u(0,0)\ne 0$ and $\delta \neq 0$. Let us write $\Phi_{F,G}(x,y)= (f(x,y),yA(x,y))$ and $\Phi_{F,H}(x,y)= (g(x,y),yB(x,y))$; we have $A(0,0)\neq 0$, $B(0,0)\ne 0$ and $(f(x,0),0)=(g(x,0),0)= \phi(x)$. 

The foliations $\mathcal I^*=\Phi_{F,G}^*{\mathcal I}$ and $\mathcal L^*=\Phi_{F,H}^*{\mathcal L}$ are defined as
\begin{eqnarray*}
yAu\,d(yA-f^2) - (yA-f^2 +\delta)\,d(yAu)&=&0,\\
yBu\,d(yB-g^2) - (yB-g^2 + a+ \delta)\,d(yBu)&=&0
\end{eqnarray*}
We see easily that the curve of tangencies is given by $ABau^2{\phi}{\phi^{\prime}}y + y^2(...)=0$, so that the component   different from $C=\{y=0\}$ crosses $C$ at $p$ transversaly.
\end{proof}
We proceed now to examine the points of tangency between $\mathcal I^*$ and $\mathcal L^*$   that possibly appear at ${\mathcal Z}_1(\mathcal I^*,C)\cup {\mathcal Z}_1(\mathcal L^*,C)$. If we denote their number as $tang_1 (\mathcal I^*,\mathcal L^*)$, we have seen that 
$$
tang_1 (\mathcal I^*,\mathcal L^*) = Z_1(\mathcal I^*,C)+ Z_1(\mathcal L^*,C).
$$
In fact, we have seen that out of ${\mathcal Z}_1(\mathcal I^*,C)\cup {\mathcal Z}_1(\mathcal L^*,C)$ tangency curves correspond to each other when restricted to $C$ and $C_0$.

{\bf We claim that this equality holds at each point of} ${\mathcal Z}_1(\mathcal I^*,C)\cup {\mathcal Z}_1(\mathcal L^*,C)$. Let us consider some point $p\in {\mathcal Z}_1(\mathcal I^*,C)$; since $\Phi_{F,G}$ is not a local biholomorphism, we have as explained before the possibilities $\bf A1$, $\bf A2$ and $\bf B$, the first two occuring when $\phi(p)\in tang({\mathcal F}_0,{\mathcal G}_0 \cap C_0$. If $p$ is $\bf A1$ or $\bf A2$ for $\Phi_{F,G}$, then $p$ is $\bf B$ for $\Phi_{F,H}$ (in the same way, when $q\in {\mathcal Z}_1(\mathcal L^*,C)$ is $\bf A1$ or $\bf A2$ for $\Phi_{F,H}$, then $q$ is $\bf B$ for $\Phi_{F,G}$. It may happen also that $p$ is $\bf B$ for $\Phi_{F,G}$ and $\Phi_{F,H}$. The reason is that we are supposing the submersions $F$, $G$ and $H$ to be independent so that we are in case $\bf A$ for maps $\Phi_{F,G}$ and $\Phi_{F,H}$ simultaneously.

{\bf Case 1: $p\in {\mathcal Z}_1(\mathcal  I^*,C)$ is ${\bf A1}$ for $\Phi_{F,G}$ and $\bf B$ for $\Phi_{F,H}$.} The local equations for $\mathcal I^*$ and $\mathcal L^*$ at $p$ are
\begin{eqnarray*}
(x+b)\xi(ydv+vdy)-d.yv \{(x+b)d\xi+{\xi}dx\}&=&0\\
(x+b^{'})(v^{'}dy+ydv^{'})-d.yv^{'}dx&=&0.
\end{eqnarray*}
The line of tangencies has equation
$$
(b-b^{'}){\xi}vv^{'}-(x+b)(x+b^{'})[vv^{'}{\xi}_x+{\xi}(v^{'}v_x-vv^{'}_x)]=0
$$
We observe that $mult_0({\xi}vv^{'})=mult_0(v)+1+mult_0(v^{'})=Z(\mathcal I^*,C,p)+Z(\mathcal L^*,C,p)$ (it may happen $Z(\mathcal L^*,C,p)=0$).

{\bf Case 2: $p\in {\mathcal Z}_1(\mathcal  I^*,C)$ is $\bf A2$  for $\Phi_{F,G}$ and  $\bf B$ for $\Phi_{F,H}$.} The local equations are
\begin{eqnarray*}
(x+b)(ydv+vdy)-d.yvdx&=&0\\
(x+b^{'})(v^{'}dy+ydv^{'})-d.yv^{'}dx&=&0
\end{eqnarray*}
The line of tangencies has equation
$$
(b-b^{'})vv^{'}-(x+b)(x+b^{'})[v^{'}v_x-vv^{'}_x]=0
$$
We remark that $mult_0(vv^{'})=mult_0(v)+mult_0(v^{'})=Z(\mathcal I^*,C,p)+Z(\mathcal L^*,C,p)$ (it may happen that $Z(\mathcal L^*,C,p)=0$).

{\bf Case 3: $p\in {\mathcal Z}_1(\mathcal  I^*,C)$ is $\bf B$  for $\Phi_{F,G}$ and  $\Phi_{F,H}$.} The conclusion is
the same as above: $mult_0(vv^{'})=mult_0(v)+mult_0(v^{'})=Z(\mathcal I^*,C,p)+Z(\mathcal L^*,C,p)$.
(it may happen $Z(\mathcal L^*,C,p)=0$).

The remaining cases (when $p\in {\mathcal Z}_1(\mathcal L^*,C,p)$)\,: $p$ is $\bf A1$ for $\Phi_{F,H}$ and $\bf B$ for $\Phi_{F,G}$; $p$ is $\bf A2$ for $\Phi_{F,H}$ and $\bf B$ for $\Phi_{F,G}$;\  $p$ is $\bf B$  for both $\Phi_{F,H}$ and  $\Phi_{F,G}$ are entirely similar.

We conclude from $tang_1 (\mathcal I^*,\mathcal L^*) = Z_1(\mathcal I^*,C)+ Z_1(\mathcal L^*,C)$ (and the fact that $b$, $b'$ are generic) that the terms $[vv^{'}{\xi}_x+{\xi}(v^{'}v_x-vv^{'}_x)]$ (first case) and $[v^{'}v_x-vv^{'}_x]$ (second and third cases) have the same multiplicities at $0$ as ${\xi}vv^{'}$ and $vv^{'}$ respectively; {\bf the claim is proved}. 

Let us make explicit the relations between the several multiplicities involved before.
\begin{itemize} 
\item Case 1: we write $v(x,0)=ax^l+\dots$ and $v^{'}(x,0)=cx^m+\dots$. It follows that $[vv^{'}{\xi}_x+{\xi}(v^{'}v_x-vv^{'}_x)]=ac(1-m+l)x^{m+l}+\dots$. Since $mult_0({\xi}vv^{'})=m+l+1$ necessarily $m=l+1$. Using $v(x,0)=x^{d-1}A_1(x,0)$ (and $v^{'}(x,0)=B^{'}(x,0)$) we get: 
$$
mult_0(A_1(x,0))+d = mult_0(B^{'}(x,0)) 
$$
\item Case 2: we write again $v(x,0)=ax^l+\dots$ and $v^{'}(x,0)=cx^m+\dots$. Then $[v^{'}v_x-vv^{'}_x)]=ac(l-m)x^{m+l-1}+\dots $. Since $mult_0(vv^{'})=m+l$, we see that $l=m$. Using $v(x,0)=A_2(x,0)$ and $v^{'}(x,0)=B^{'}(x,0)$
$$
mult_0(A_2(x,0))=mult_0(B^{'}(x,0))
$$
\item Case 3: it is analogous to Case 2 and we find
$$
mult_0(B(x,0))=mult_0(B^{'}(x,0))
$$
\end{itemize} 

There are correspondent equalities when  $p$ is $\bf A_1$ for $\Phi_{F,H}$ and $\bf B$ for $\Phi_{F,G}$\,(\,$mult_0(A^{'}_1(x,0))+d=mult_0(B(x,0))$) or $p$ is $\bf A_2$ for $\Phi_{F,H}$ and $\bf B$ for $\Phi_{F,G}$\,(\,$mult_0(A^{'}_2(x,0))=mult_0(B(x,0)$).

\section{Proof of Theorem \ref{main-thm}}\label{sec-proof-thm}
Let us take some $C^{\infty}$ pertubation $\tilde C$ of $C$ and look to the curves $\Phi_{F,G}(\tilde C)$ and $\Phi_{F,H}(\tilde C)$, which are $C^{\infty}$ pertubations of $C_0$; we ask $\tilde C$ to be a holomorphic smooth curve with $({\tilde C}.{C})_p=1$ when passing through each $p\in {\mathcal Z_1}(\mathcal I^*,C) \cup {\mathcal Z_1}(\mathcal L^*,C)$   and ask also that   $\Phi_{F,G}$ and $\Phi_{F,H}$   be holomorphic along these (local) holomorphic curves. Let us observe again that $\Phi_{F,G}$ is not a local biholomorphism at a point $p\in {\mathcal Z_1}(\mathcal I^*,C)$ (and $\Phi_{F,H}$ is not a local biholomorphism at a point $p\in {\mathcal Z_1}(\mathcal L^*,C)$ as well). 

We proceed now to prove that for any $p\in {\mathcal Z}_1(\mathcal  I^*,C)\cup {\mathcal Z}_1(\mathcal L^*,C)$ one has 
$$
\sum_q(\Phi_{F,G}({\tilde C}).C_0)_q + (\Phi_{F,H}({\tilde C}).C_0)_q \geq 2
$$
for $q$ close to $\phi(p)$. Observe that in principle this number should be equal to $({\tilde C}.C)_{p}+ ({\tilde C}.C)_{P}=2$.  Let us go back to the cases we discussed in the last Section.

\begin{itemize}
\item $p\in {\mathcal Z}_1(\mathcal I^*,C)$ is $\bf A1$ for $\Phi_{F,G}$ and $\bf B$ for $\Phi_{F,H}$. The pertubation $\tilde C$ near $p$ has to be contained in some small sector around $C$, where $\Phi_{F,G}$ is holomorphic. Since 
$$
\Phi_{F,G}=\left(x,\dfrac{yA_1(x,y)+xh(\xi(x,y))}{x+yA_1(x,y)}\right)
$$
when we put $y={\epsilon}x$ (for $\tilde C$) we see that $\sum_{q} (\Phi_{F,G}({\tilde C}).C_0)_q$ is the number of solutions (near $\phi(p)$) to the equation
$$
\dfrac{\epsilon\,x\,A_1(x,\epsilon\,x)+xh(\xi(x,\epsilon\,x))}{x+\epsilon\,x\,A_1(x,\epsilon\,x)}=h(x)
$$
which is $= mult_0(A_1(x,0))$. In order to estimate  $\sum_{q} (\Phi_{F,H}({\tilde C}).C_0)_q$ we use
$$
\Phi_{F,H}(x,y)=(x,yB^{'}(x,y)+h(\xi(x,y)))
$$
and we have to find the number of solutions of 
$$
\epsilon\,x\,B^{'}(x,\epsilon\,x)+h(\xi(x,\epsilon\,x))=h(x)
$$
(remember that now $p$ is of $B$ type for $\Phi_{F,H}$), which is readily seen to be $1+mult_0(B^{'}(x,0))$.
We have seen before that  $mult_0(B^{'}(x,0))= mult_0(A_1(x,0))+d$, so for $q$ close to $\phi(p)$:
$$
\sum_q(\Phi_{F,G}({\tilde C}).C_0)_q + (\Phi_{F,H}({\tilde C}).C_0)_q= 2mult_0(A_1(x,0))+d+1
$$  
which is strictly bigger than 2 when $d>1$.
\item  $p\in {\mathcal Z}_1(\mathcal I^*,C)$ is $\bf A2$ for $\Phi_{F,G}$ and $\bf B$ for $\Phi_{F,H}$. We have
$$
\Phi_{F,G}(x,y)=\left(x, \dfrac {yA_2(x,y) +h(\xi(x,y))}{1+yA_2(x,y)}\right)
$$
and 
$$
\Phi_{F,H}(x,y)=(x,yB^{'}(x,y)+h(\xi(x,y)))
$$
Using again $y=\epsilon x$, we get $\sum_q(\Phi_{F,G}({\tilde C}).C_0)_q=1+mult_0A_2(x,0)$ and $\sum_q(\Phi_{F,H}({\tilde C}).C_0)_{\phi(q)}=1+mult_0(B^{'}(x,0))$
Therefore for  $q$ close to $\phi(p)$:
$$
\sum_q(\Phi_{F,G}({\tilde C}).C_0)_q + (\Phi_{F,H}({\tilde C}).C_0)_q= 2+2mult_0(A_2(x,0)) 
$$
\item $p\in {\mathcal Z}_1(\mathcal I^*,C)$ is $\bf B$ for $\Phi_{F,G}$ and $\bf B$ for $\Phi_{F,H}$. Similarly we find for $q$ close to $\phi(p)$:
$$
\sum_q(\Phi_{F,G}({\tilde C}).C_0)_q + (\Phi_{F,H}({\tilde C}).C_0)_q= 2+2mult_0(B(x,0)) 
$$
\end{itemize}

The remaining cases are analogous. We conclude that Case ${\bf A1}$ never appears  and that Cases $\bf A2$ and $\bf B$ are present only at points $p$ where $\Phi_{F,G}$ is a local biholomorphisms.

\section{Appendix: Automorphisms of a plane curve}\label{appendix}

We present a proof of the following theorem given by Jose Felipe Voloch.

\begin{thm}\label{Theorem1}
Let $C$ be a smooth plane curve of degree $d\geq 3$, then every automorphism of $C$ is linear, i.e. it comes from an element of $Aut(\mathbb{P}^2)$.
\end{thm}

The case $d=3$ is a consequence of Legendre's normal form, see  \cite{HOU}, so we focus on the case $d \geq 4$. Before proving the theorem we give some useful remarks and lemmas based on exercises $17$ and $18$ of \cite{AR}. We say that the set $S=\{p_1, \ldots, p_k\} \subseteq \mathbb{P}^2$ of distinct points \textbf{impose independent conditions on curves of degree $n$} if  $h^0(\mathbb{P}^2, \mathcal{I}_S (n)) = h^0(\mathbb{P}^2, \mathcal{O}(n)) - k$.

\begin{lemma}
Any set of $n+1$ points impose independent conditions on curves of degree $n$. On the other hand ${n+2}$ points impose independent conditions if and only if they are not aligned.
\end{lemma}

\begin{proof}
Take first $S=\{p_1, \ldots, p_{n+1}\}$ and denote $S_k =\{p_1, \ldots, p_k\}$. Taking the product of $n$ lines through another point we see that $H^0(\mathbb{P}^2, \mathcal{I}_{S_{i+1}} (n))$ is strictly contained in $H^0(\mathbb{P}^2, \mathcal{I}_{S_{i}} (n))$, therefore $h^0(\mathbb{P}^2, \mathcal{I}_{S} (n))= h^0(\mathbb{P}^2, \mathcal{O}(n)) - (n+1)$. 

Consider now a set $S=\{p_1, \ldots, p_{n+1}, p_{n+2} \}$. If they are over a line $L$ and $E$ is a curve of degree $n$ passing through $n+1$ of them Bezout's theorem implies $L \subseteq E$. This shows that $S$ fails to impose independent conditions on curves of degree $n$. Suppose now that every curve of degree $n$ passing by $n+1$ points contains also the other point of $S$. If they are not aligned, we can take for example the curve $E$ formed by lines joining $p_{n+1}$ with points $p_1, \ldots, p_{n}$, thus $p_{n+2}$ must be on this curve and we can assume that $p_n$, $p_{n+1}$ and $p_{n+2}$ are aligned. If some $p_j$, $j=1, \ldots, n-1$ is not on this line we consider $E'$ obtained from $E$ replacing $\overline{p_{n+1}, p_j}$ by a generic line passing by $p_{n+1}$, thus $E'$ contains $(n+1)$ points but not $S$, contradiction. 
\end{proof}

Let $D$ be an effective divisor on $C$ of degree $m$. We use previous lemma in order to study meromorphic functions on $C$ having $D$ as polar divisor. Changing the fiber if necessary we will assume from now that $D$ has not multiple points. We recall that $l(D)$ is the dimension of the space of meromorphic functions f such that $(f) +D \geq 0$ and $i(D)$ is the dimension of the space of holomorphic forms $\omega$ such that $(\omega) \geq D$.

\begin{prop}
If $m \leq d-2$ then $l(D)=1$.
\end{prop}
\begin{proof}
Recall (see \cite{R}) that holomorphic $1-$forms on $C=\{P=0\}$ are generated by elements $\frac{x^i y^j}{P_y}dx $ with $i+j \leq d - 3$. By the previous lemma the dimension of the space of polynomials vanishing at $D$ is $g(C) - m$, thus $i(D) = g(C) -m$ and Riemann-Roch gives $l(D) = 1$.
\end{proof}

\begin{prop}
If $m=d-1$ and $l(D) \geq2$ then $D= E- p$ where $p \in C$ and $E \in |\mathcal{O}_C(1)|$.
\end{prop}
\begin{proof}
Once again Riemann-Roch theorem gives $i(D) = g-d + l(D) \geq g-(d-1)+1$ then points of $D$ do not impose independent conditions and they must be aligned. We conclude by noting that intersection of a line with $C$ is a divisor of degree $d$. 
\end{proof}

Finally we have 

\begin{prop}
$|\mathcal{O}_C(1)|$ is the only linear system of degree $d$ and dimension $3$.
\end{prop}
\begin{proof}
Let $D \in |\mathcal{O}_C(1)|$ be an aligned divisor of degree $d$ on $C$. Then points of $D$ fails to impose independent conditions on curves of degree $d-3$ and $i(D) = g(C) - (d-2)$ or equivalently $l(D) = 3$. If $A$ is another effective divisor of degree $d$ and $l(A)=3$ then any subset of $d-1$ points are aligned. We conclude that $A \in |\mathcal{O}_C(1)|$ and is linearly equivalent to $D$. 
\end{proof}

\textbf{Proof of Theorem \ref{Theorem1}:} Let $\phi : C \rightarrow C$ be an automorphism of $C$. Last proposition implies that for any line $L$ on $\mathbb{P}^2$, points of $\phi(L \cap C)$ determine a line $L' \subseteq \mathbb{P}^2$, thus $\phi$ comes from an automorphism of $\check{\mathbb{P}}^2$ which corresponds to an element of $Aut(\mathbb{P}^2)$.


\begin{thebibliography}{99}
\frenchspacing

\bibitem {FA}
M. Falla Luza and F. Loray. \emph{On the number of fibrations transverse to a rational curve in complex surfaces}, Comptes Rendus Math\'ematique, v. 354, pg. 470-474, 2016. 

\bibitem {SA}
P. Sad. \emph {Projective transverse structures for some foliations}, Trends in Mathematics (Singularities in Geometry, Topology, Foliations and Dynamics, pg. 197-206, 2017), Birkhauser.

\bibitem {AR}
E. Arbarello, M. Cornalba, P. Griffiths, J. Harris. Geometry of Algebraic Curves, Volume 1, Springer-Verlag, 1984.

\bibitem {BRU}
M. Brunella. Birational Geometry of Foliations. Springer-Verlag, New York, 2015. 

\bibitem {HOU}
D. Husem\"oller. Elliptic Curves, Springer Graduate Texts in Mathematics, 1987.

\bibitem {GRA}
H. Grauert. \emph{\"Uber Modifikationen and exzeptionelle analytische Mengen}, Math. Ann. 146, pg. 331-368, 1962.

\bibitem {SAV}
V.I. Savelev. \emph{Zero-type imbedding of a sphere into complex surfaces}, Moscow Univ. Math. Bull. 37, no. 4, pg. 34-39, 1982.

\bibitem {UE}
T.Ueda. \emph{On the neighborhood of a compact complex curve with topologically trivial normal bundle}, J. Math. Kyoto Univ. 22, pg. 583-607, 1982.

\bibitem {MEZ}
R. Meziani and P. Sad. \emph{Singularités nilpotentes et intégrales premières}, Publ. Mat. 51, N 1, 143 -161,2007. 

\bibitem{R}
{\sc E. Reyssat},
\emph{Quelques Aspects des Surfaces De Riemann }  Birkhauser Verlag AG.

\end{thebibliography}
\end{document}